\documentclass{amsart}

\makeatletter
\@namedef{subjclassname@2020}{%
  \textup{2020} Mathematics Subject Classification}
\makeatother 

\usepackage{wasysym}
\usepackage{amsthm,amssymb,amsfonts,latexsym,mathtools,thmtools}
\usepackage[T1]{fontenc}
\usepackage{tikz-cd} % Commutative diagrams.
\usepackage{enumitem} % Customization of items.
\usepackage{hyperref} %hyperlinks
\hypersetup{
    colorlinks=true,
    linkcolor=blue,
    filecolor=blue,      
    urlcolor=cyan,
    linktocpage=true
}

\newtheorem{theorem}{Theorem}[section]
\newtheorem{lemma}[theorem]{Lemma}

\theoremstyle{definition}
\newtheorem{definition}[theorem]{Definition}
\newtheorem{proposition}[theorem]{Proposition}
\newtheorem{corollary}[theorem]{Corollary}
\newtheorem{remark}[theorem]{Remark}
\newtheorem{example}[theorem]{Example}

\theoremstyle{remark}

\numberwithin{equation}{section}

\begin{document}

\title{\textbf{On $\Sigma$-skew reflexive-nilpotents-property for rings}}

%    Remove any unused author tags.

%    author one information

\author{H\'ector Su\'arez}
\address{Universidad Pedag\'ogica y Tecnol\'ogica de Colombia, Sede Tunja}
\curraddr{Campus Universitario}
\email{hector.suarez@uptc.edu.co}

\author{Sebasti\'an Higuera}
\address{Universidad Nacional de Colombia - Sede Bogot\'a}
\curraddr{Campus Universitario}
\email{sdhiguerar@unal.edu.co}
\thanks{}

\author{Armando Reyes}
\address{Universidad Nacional de Colombia - Sede Bogot\'a}
\curraddr{Campus Universitario}
\email{mareyesv@unal.edu.co}
\thanks{}

%    author two information

\thanks{The authors were supported by the research fund of Department of Mathematics, Faculty of Science, Universidad Nacional de Colombia - Sede Bogot\'a, Colombia, HERMES CODE 52464.}



\date{}

\dedicatory{Dedicated to the memory of Professor Vyacheslav Artamonov}

\begin{abstract}

In this paper, we study the reflexive-nilpotents-property (briefly, RNP) for Ore extensions of injective type, and more generally, skew PBW extensions. With this aim, we introduce the notions of $\Sigma$-skew CN rings and $\Sigma$-skew reflexive (RNP) rings, for $\Sigma$ a finite family of ring endomorphisms of a ring $R$. Under certain conditions of compatibility, we study the transfer of the $\Sigma$-skew RNP property from a ring of coefficients to an Ore extension or skew PBW extension over this ring. We also consider this property for localizations of these noncommutative rings. Our results extend those corresponding presented by Bhattacharjee \cite{Bhattacharjee2020}.

\end{abstract}

\maketitle

%\tableofcontents

\section{Introduction}\label{sect.prelim}

Throughout the paper, $\mathbb{N}$ and $\mathbb{Z}$ denote the set of natural numbers including zero and the ring of integers, respectively. The symbol $\Bbbk$ denotes a field, and $\Bbbk^{*}:= \Bbbk\  \backslash\ \{0\}$. Every ring is associative with identity unless otherwise stated. For a ring $R$, $N(R)$ denotes its set of nilpotent elements, $N_{*}(R)$ is its prime radical, and $N^{*}(R)$ is its upper nilradical (i.e. the sum of all nil ideals). It is well-known that $N^{*}(R)\subseteq N(R)$, and if the equality $N^{*}(R) = N(R)$ holds, then $R$ is called {\em NI} (Marks \cite{Marks2001}). Note that by definition, $R$ is NI if and only if $N(R)$ forms an ideal if and only if $R/N^{*}(R)$ is reduced, that is, $R/N^{*}(R)$ has no nonzero nilpotent elements. Hong and Kwak \cite[Corollary 13]{HongKwak2000}, showed that a ring $R$ is NI if and only if every minimal strongly prime ideal of $R$ is completely prime. If the equality $N_{*}(R) = N(R)$ holds, then $R$ is called 2-\textit{primal} (Birkenmeier et al. \cite{Birkenmeieretal1993}); equivalently, $N^{*}(R)$ is a completely semiprime ideal of $R$. 2-primal rings are clearly NI. The converse of this implication need not hold, but if $R$ is an NI ring of bounded index nilpotency, then $R$ is 2-primal \cite[Proposition 1.4]{Hwangetal2006}.

\vspace{0.2cm}

Following Cohn \cite{Cohn1999}, $R$ is said to be {\em reversible} if $ab = 0$ implies $ba = 0$, for $a, b\in R$. Commutative rings and reduced rings are clearly reversible. Reversible rings were studied under the name {\em zero commutative} by Habeb \cite{Habeb1990}. As a matter of fact, the class of NI rings contains nil rings and reversible rings. Lambek \cite{Lambek1971} called a ring $R$ {\em symmetric} provided $abc = 0$ implies $acb = 0$, for $a, b, c\in R$. Of course, commutative rings are symmetric, and symmetric rings are reversible, but the converses do not hold (\cite[Examples I.5 and II.5]{AndersonCamillo1999}, and \cite[Examples 5 and 7]{Marks2002}). We know that every reduced ring is symmetric \cite[Lemma 1.1]{Shin1973}, but the converse does not hold \cite[Example II.5]{AndersonCamillo1999}. Bell \cite{Bell1970} used the
term \emph{Insertion-of-Factors-Property} ({\em IFP}) for a ring $R$ if $ab = 0$ implies $aRb = 0$, for $a, b\in R$ (Narbonne \cite{Narbonne1982} and Habeb \cite{Habeb1990} used the terms {\em semicommutative} and {\em zero insertive}, respectively). Some results about IFP rings are due to Shin \cite{Shin1973}. Note that every reversible ring is IFP, but the converse need not hold \cite[Lemma 1.4 and Example 1.5]{KimLee2003}. Reversible rings are reflexive, and there exists a reflexive and IFP ring which is not symmetric \cite[Examples 5 and 7]{Marks2002}. In fact, a ring $R$ is reflexive and IFP if and only if $R$ is reversible \cite[Proposition 2.2]{KwakLee2012}. It is easy to see that IFP rings are 2-primal. Note that IFP rings are Abelian (every idempotent is central). 

\vspace{0.2cm}

In addition to reversible rings and their aforementioned generalizations of reduced rings, we also have the notion of reflexive ring. Mason \cite{Masonreflexiveideals} introduced the reflexive property for right ideals by defining a right ideal $I$ of a ring $R$ as \textit{reflexive} if for $a, b \in R$, $aRb \subseteq I$ implies $bRa \subseteq I$, and a ring $R$ is called \textit{reflexive} if the zero ideal of $R$ is reflexive (i.e. $aRb = 0$ implies $bRa = 0$, for $a, b\in R$). Equivalently, $R$ is reflexive if and only if $IJ = 0$ implies $JI = 0$, for ideals $I, J$ of $R$ \cite[Lemma 2.1]{KwakLee2012}. By a direct computation one can check that semiprime rings and reversible rings are reflexive, and every ideal of a fully idempotent ring $R$ (i.e. $I^2=I$, for every ideal of $R$) is reflexive (Courter \cite{Courter1969}). By \cite[Example 2.3]{KwakLee2012}, we know that the IFP and the reflexive ring properties do not imply each other. Kwak and Lee \cite{KwakLee2012} characterized the aspects of the reflexive and one-sided idempotent reflexive properties. They established a method by which a reflexive ring, which is not semiprime, can always be constructed from any semiprime ring, and showed that the reflexive property is Morita invariant. In the literature, different generalizations of reflexive rings have been formulated. Let us recall them.

\vspace{0.2cm}

Kim \cite{Kim2005} introduced the notion of idempotent reflexive ring as a generalization of reflexive ring. For a one-sided ideal $I$ of a ring $R$, $I$ is called {\em right idempotent reflexive} if $aRe\subseteq I$ implies $eRa\subseteq I$, for any $a, e^2 = e\in R$, and the ring $R$ is called {\em right idempotent reflexive} if the zero ideal is a right idempotent reflexive ideal. {\em Left idempotent reflexive ideals} and {\em left idempotent reflexive rings} are defined similarly. If a ring $R$ is both left and right idempotent reflexive, then $R$ is called an {\em idempotent reflexive ring}. For more details about these rings, see Kim and Baik \cite{KimBaik2006}. As one can check, reflexive rings and Abelian rings are idempotent reflexive. By \cite[Example 2.3(1)]{KwakLee2012}, we know that there exists an idempotent reflexive ring which is not reflexive. Similar to the case of reflexive rings, $R$ is right idempotent reflexive if and only if $IJ = 0$ implies $JI = 0$, for all right ideals $I, J$ of $R$ where $J$ is a right ideal generated by a subset of idempotents in $R$, if and only if $IJ = 0$ implies $JI = 0$, for all ideals $I, J$ of $R$ where $J$ is an ideal generated by a subset of idempotents in $R$ \cite[Lemma 3.4]{KwakLee2012}.

\vspace{0.2cm}

Kheradmand et al. \cite{Kheradmandetal2017} also introduced a generalization of reflexive rings. A ring $R$ is said to be \textit{RNP} ({\em reflexive-nilpotents-property}) if for $a,b \in N(R)$, $aRb = 0$ implies $bRa = 0$. Of course, reflexive rings are RNP but the converse need not hold \cite[Example 1.2(1)]{Kheradmandetal2017(b)}. In the same year, Kheradmand et al. \cite{Kheradmandetal2017(b)} defined a more general class than reflexive and RNP rings by considering nil ideals, and called a ring $R$ \textit{nil-reflexive} if $IJ = 0$ implies $JI = 0$, for nil ideals $I,J$ of $R$. As it is clear, reflexive rings are RNP and RNP rings are nil-reflexive, but the converse need not hold \cite[Example 1.2]{Kheradmandetal2017(b)}. They also showed that $R$ is a nil-reflexive ring if and only if $aRb = 0$ implies $bRa = 0$, for elements $a, b \in N^{*}(R)$ \cite[Proposition 2.1]{Kheradmandetal2017(b)}. Notice that the concepts of (nil)reflexive rings and NI rings are independent of each other \cite[Examples 1.5 and 2.9]{Kheradmandetal2017(b)}. Nevertheless, if $R$ is an NI ring, then the following assertions are equivalent: $R$ is nil-reflexive; $aRb = 0$, for $a, b\in N(R)$, implies $bRa = 0$; $IJ = 0$ implies $JI = 0$, for all nil right (or, left) ideals $I, J$ of $R$ \cite[Proposition 1.6]{Kheradmandetal2017(b)}.

\vspace{0.2cm}

In addition to the reduced rings and its generalizations described above, they have all been extended by ring endomorphisms. According to Krempa \cite{Krempa1996}, an endomorphism $\sigma$ of a ring $R$ is called {\em rigid} if $a\sigma(a) = 0$ implies $a=0$, for $a\in R$, and by Hong et al. \cite{Hongetal2000}, $R$ is called $\sigma$-{\em rigid} if there exists a rigid endomorphism $\sigma$ of $R$. Any rigid endomorphism of a ring $R$ is a monomorphism and $\sigma$-rigid rings are reduced rings \cite[Proposition 5]{Hongetal2000}. Following Ba\c{s}er et al. \cite{Baseretal2009} an endomorphism $\sigma$ of a ring $R$ is called {\em right skew reversible} if whenever $ab = 0$, for $a, b\in R$, $b\sigma(a) = 0$, and the ring $R$ is called {\em right} $\sigma$-{\em skew reversible} if there exists a right skew reversible endomorphism $\sigma$ of $R$. Similarly, left $\sigma$-skew reversible rings are defined. A ring $R$ is said to be $\sigma$-{\em skew reversible} if it is both right and left $\sigma$-reversible. It is important to say that $R$ is a $\sigma$-rigid ring if and only if $R$ is semiprime and right $\sigma$-skew reversible for a monomorphism $\sigma$ of $R$ \cite[Proposition 2.5 (iii)]{Baseretal2009}.

\vspace{0.2cm}

Kwak et al. \cite{Kwaketal2014} extended the reflexive property to the skewed reflexive property by ring endomorphisms. An endomorphism $\sigma$ of a ring $R$ is called \textit{right} (resp., \textit{left}) \textit{skew reflexive} if for
$a, b \in R$, $aRb = 0$ implies $bR\sigma(a) = 0$ (resp., $\sigma(b)Ra = 0$), and $R$ is called \textit{right} (resp., \textit{left}) \textit{$\sigma$-skew reflexive} if there exists a right (resp., left) skew reflexive endomorphism $\sigma$ of $R$. $R$ is said to be \textit{$\sigma$-skew reflexive} if it is both right and left $\sigma$-skew reflexive. It is clear that $\sigma$-rigid rings are right $\sigma$-skew reflexive. More precisely, a ring $R$ is reduced and right $\sigma$-skew reflexive for a monomorphism $\sigma$ of $R$ if and only if $R$ is $\sigma$-rigid \cite[Theorem 2.6]{Kwaketal2014}. Now, Bhattacharjee \cite{Bhattacharjee2020} extend the notion of RNP rings to ring endomorphisms $\sigma$ and introduced the notion of $\sigma$-skew RNP rings as a generalization of $\sigma$-skew reflexive rings. An endomorphism $\sigma$ of a ring $R$ is called \textit{right} (resp., \textit{left}) \textit{skew} RNP if for $a, b \in N(R)$, $aRb = 0$ implies $bR\sigma(a) = 0$ (resp., $\sigma(b)Ra = 0$). A ring $R$ is called \textit{right} (resp., \textit{left}) \textit{$\sigma$-skew} RNP if there exists a right (resp., left) skew RNP endomorphism $\sigma$ of $R$. $R$ is said to be \textit{$\sigma$-skew} RNP if it is both right and left $\sigma$-skew RNP. From \cite[Remark 1.2]{Bhattacharjee2020}, we know that reduced rings are $\sigma$-skew RNP for any endomorphism $\sigma$, and every right (resp., left) $\sigma$-skew reflexive ring is right (resp., left) $\sigma$-skew RNP. By \cite[Example 1.3]{Bhattacharjee2020}, we have that the notion of $\sigma$-skew RNP ring is not left-right symmetric. However, if $R$ is an RNP ring with an endomorphism $\sigma$, then $R$ is right $\sigma$-skew RNP if and only if $R$ is left $\sigma$-skew RNP. 

\vspace{0.2cm}

Taking into account our interest in noncommutative rings defined by endomorphisms, in this paper we will focus our attention on the study of ring-theoretical notions above for the skew polynomial rings (also known as Ore extensions) defined by Ore \cite{Ore1933}, and the skew PBW extensions introduced by Gallego and Lezama \cite{LezamaGallego}. As is well-known, skew polynomial rings are one of the most important families of noncommutative rings of polynomial type related with the study of quantum groups, differential operators, noncommutative algebraic geometry and noncommutative differential geometry (e.g. Brown and Goodearl \cite{BrownGoodearl2002}, Goodearl and Warfield \cite{GoodearlWarfield2004} or McConnell and Robson \cite{McConnellRobson2001}), and a lot of papers have been published with the aim of studying different theoretical properties of these objects. Now, regarding skew PBW extensions, their importance is that these objects generalize PBW extensions defined by Bell and Goodearl \cite{BellGoodearl1988}, families of differential operator rings, Ore extensions of injective type, several algebras appearing in noncommutative algebraic geometry, examples of quantum groups, and other families of noncommutative rings having PBW bases. Since its introduction, ring-theoretical and homological properties of skew PBW extensions have been studied by some people (e.g. Artamonov \cite{Artamonov2015}, Hamidizadeh et al. \cite{Hamidizadehetal2020}, Hashemi et al.  \cite{HashemiKhalilAlhevaz2017,HashemiKhalilAlhevaz2019,HashemiKhalilGhadiri2019}, Lezama et al.  \cite{Lezama2020,Lezama2021,LezamaGallego2016,LezamaGomez2019}, Tumwesigye et al. \cite{TumwesigyeRichterSilvestrov2019}, and the authors \cite{HigueraReyes2021,ReyesSuarez2019-1}, and \cite{Suarez2017}). As a matter of fact, a book that includes several of the works carried out for these extensions has been published (Fajardo et al. \cite{Fajardoetal2020}).

\vspace{0.2cm}

The paper is organized as follows. Section \ref{SPBW} introduces notation, and gives a brief account of the main results on skew PBW extensions appearing in \cite{Fajardoetal2020,LezamaGallego} which are required in the paper. In Section \ref{sect-Ore ext}, the notion of $\Sigma$-skew reflexive is introduced for a family $\Sigma$ of endomorphisms of a ring $R$ in order to generalize the notions of skew reflexive endomorphism and $\sigma$-skew reflexive ring, and to study different properties for this new family of rings. In Section \ref{sect-NocSigma}, we introduce the concept of skew RNP for a finite family of endomorphisms $\Sigma$ of a ring $R$, and the $\Sigma$-skew RNP rings as a generalization of $\sigma$-skew reflexive rings and $\sigma$-skew RNP rings, presented \cite{Kwaketal2014} and \cite{Bhattacharjee2020}, respectively. We study properties relating the $\Sigma$-skew RNP condition with different kinds of rings widely studied in the literature. Also, we present some results that generalize in some way the results presented in \cite{Bhattacharjee2020} for skew polynomial rings. Finally, we study the behavior of the $\Sigma$-skew RNP property for Ore localization by regular elements. In particular, we present a theorem that characterizes this property for the localization of skew PBW extensions.
%%%%%%%%%%%%%%%%%%%%%%%%%%%%%%%%%%%%%%%%%
%%%%%%%%%%%%%%%%%%%%%%%%%%%%%%%%%%%%%%%%%
\section{Preliminaries}\label{SPBW}

\begin{definition}(\cite[Definition 1]{LezamaGallego}) \label{def.skewpbwextensions}
Let $R$ and $A$ be rings. We say that $A$ is a \textit{skew PBW extension over} $R$ (the ring of coefficients), denoted $A=\sigma(R)\langle
x_1,\dots,x_n\rangle$, if the following conditions hold:
\begin{enumerate}
\item[\rm (1)]$R$ is a subring of $A$ sharing the same identity element.
\item[\rm (2)] there exist finitely many elements $x_1,\dots ,x_n\in A$ such that $A$ is a left free $R$-module, with basis the
set of standard monomials
\begin{center}
${\rm Mon}(A):= \{x^{\alpha}:=x_1^{\alpha_1}\cdots
x_n^{\alpha_n}\mid \alpha=(\alpha_1,\dots ,\alpha_n)\in
\mathbb{N}^n\}$.
\end{center}
Moreover, $x^0_1\cdots x^0_n := 1 \in {\rm Mon}(A)$.
\item[\rm (3)]For each $1\leq i\leq n$ and any $r\in R\ \backslash\ \{0\}$, there exists an
    element $c_{i,r}\in R\ \backslash\ \{0\}$ such that $x_ir-c_{i,r}x_i\in R$.
\item[\rm (4)]For $1\leq i,j\leq n$, there exists $d_{i,j}\in R\ \backslash\ \{0\}$ such that
\[
x_jx_i-d_{i,j}x_ix_j\in R+Rx_1+\cdots +Rx_n,
\]
i.e. there exist elements $r_0^{(i,j)}, r_1^{(i,j)}, \dotsc, r_n^{(i,j)} \in R$ with $x_jx_i - d_{i,j}x_ix_j = r_0^{(i,j)} + \sum_{k=1}^{n} r_k^{(i,j)}x_k$.
\end{enumerate}
\end{definition}
%%%%%%%%%%%%%%%%%%%%%%%%%%%%%%%%%%%%%%%%%
From the definition it follows that every non-zero element $f \in A$ can be uniquely expressed as $f = a_0 + a_1X_1 + \cdots + a_mX_m$, with $a_i \in R$ and $X_i \in \text{Mon}(A)$, for $0 \leq i \leq m$ \cite[Remark 2]{LezamaGallego}. For $X = x^{\alpha} = x_1^{\alpha_1}\cdots x_n^{\alpha_n} \in \text{Mon}(A)$, $\text{deg}(X) = |\alpha| := \alpha_1 + \cdots + \alpha_n$.

\vspace{0.2cm}

The next result establishes the importance of the (injective) endomorphisms of the ring of coefficients. As expected, later on we will make these endomorphisms interact with the different ring-theoretical notions of rings discussed in Section \ref{sect.prelim}.

\begin{proposition}(\cite[Proposition 3]{LezamaGallego}) \label{sigmadefinition}
If $A=\sigma(R)\langle x_1,\dots,x_n\rangle$ is a skew PBW extension, then there exist an injective endomorphism $\sigma_i:R\rightarrow R$ and a $\sigma_i$-derivation $\delta_i:R\rightarrow R$ such that $x_ir=\sigma_i(r)x_i+\delta_i(r)$, for each $1\leq i\leq n$, where $r\in R$.
\end{proposition}
%%%%%%%%%%%%%%%%%%%%%%%%%
We use the notation $\Sigma:=\{\sigma_1,\dots,\sigma_n\}$ and $\Delta:=\{\delta_1,\dots,\delta_n\}$ for the families of injective endomorphisms and $\sigma_i$-derivations of Proposition \ref{sigmadefinition}, respectively. For a skew PBW extension $A = \sigma(R)\langle x_1,\dotsc, x_n\rangle$ over $R$, we say that the pair $(\Sigma, \Delta)$ is a \textit{system of endomorphisms and $\Sigma$-derivations} of $R$ with respect to $A$. For $\alpha = (\alpha_1, \dots , \alpha_n) \in \mathbb{N}^n$, $\sigma^{\alpha}:= \sigma_1^{\alpha_1}\circ \cdots \circ \sigma_n^{\alpha_n}$, $\delta^{\alpha} := \delta_1^{\alpha_1} \circ \cdots \circ \delta_n^{\alpha_n}$, where $\circ$ denotes the classical composition of functions. 
%%%%%%%%%%%%%%%%%%%%%%%%%%%%%
\begin{definition}\label{quasicommutative}
Let $A=\sigma(R)\langle x_1,\dots,x_n\rangle$ be a skew PBW extension over a ring $R$.
\begin{itemize}
    \item[{\rm (1)}] \cite[Definition 4]{LezamaGallego} $A$ is called {\it quasi-commutative} if the conditions ${\rm (3)}$ and ${\rm (4)}$ presented above are replaced by the following: 
\begin{enumerate}
    \item[(3')] For every $1 \leq i \leq n$ and $r \in R \setminus \left \{0 \right \}$ there exists $c_{i,r} \in R \setminus \left \{0 \right \}$ such that
$x_ir = c_{i,r}x_i$.
\item[(4')] For every $1 \leq i, j \leq n$ there exists $d_{i,j} \in R \setminus \left \{0 \right \}$ such that $x_jx_i = d_{i,j}x_ix_j$.
\end{enumerate}
    \item[{\rm (2)}] \cite[Definition 4]{LezamaGallego} $A$ is said to be {\it bijective} if  $\sigma_i$ is bijective for each $1 \leq i \leq n$, and $d_{i,j}$ is invertible for any $1 \leq i <j \leq n$.
    \item [\rm (3)] \cite[Definition 2.3]{AcostaLezamaReyes2015} If $\sigma_i\in \Sigma$ is the identity map of $R$ for each $i = 1, \dots , n$, (we write $\sigma_i = {\rm id}_R$), we say that $A$ is a skew PBW extension of \textit{derivation type}. Similarly, if $\delta_i\in \Delta$ is zero, for every $i$, then $A$ is called a skew PBW extension of \textit{endomorphism type}.
\end{itemize}
\end{definition}
%%%%%%%%%%%%%%%%%%%%%%%%%%%%%%%%%
\begin{remark}\label{comparisonendomorphism}
Some relationships between skew polynomial rings and skew PBW extensions are the following:
\begin{itemize}
\item \cite[Theorem 2.3]{LezamaReyes2014} If $A=\sigma(R)\langle x_1,\dotsc, x_n\rangle$ is a quasi-commutative skew PBW extension over a ring $R$, then $A$ is isomorphic to an iterated skew polynomial ring of endomorphism type.
    \item \cite[Example 5(3)]{LezamaReyes2014} If $R[x_1;\sigma_1,\delta_1]\dotsb [x_n;\sigma_n,\delta_n]$ is an iterated skew polynomial ring such that $\sigma_i$ is injective, for $1\le i\le n$; $\sigma_i(r)$, $\delta_i(r)\in R$, for $r\in R$ and $1\le i\le n$; $\sigma_j(x_i)=cx_i+d$, for $i < j$, and $c, d\in R$, where $c$ has a left inverse; $\delta_j(x_i)\in R + Rx_1 + \dotsb + Rx_n$, for $i < j$, then $R[x_1;\sigma_1,\delta_1]\dotsb [x_n;\sigma_n, \delta_n] \cong \sigma(R)\langle x_1,\dotsc, x_n\rangle$. In general, skew polynomial rings of injective type are strictly contained in skew PBW extensions. This fact is not possible for PBW extensions. For instance, the quantum plane $\Bbbk\langle x, y\rangle / \langle xy - qyx\mid q\in \Bbbk^{*} \rangle$ is an Ore extension of injective type given by $\Bbbk[y][x;\sigma]$, where $\sigma(y) = qy$, but cannot be expressed as a PBW extension. 
\item \cite[Remark 2.4 (ii)]{SuarezChaconReyes2021} Skew PBW extensions of endomorphism type are more general than iterated skew polynomial rings of endomorphism type. With the aim of illustrating the differences between these structures, we consider the situations with two and three indeterminates. If we take the iterated skew polynomial ring of endomorphism type given by $R[x;\sigma_x][y;\sigma_y]$, by definition we have the following relations: $xr = \sigma_x(r)x$, $yr = \sigma_y(r)y$, and $yx = \sigma_y(x)y$, for any element $r\in R$. On the other hand, if we have $\sigma(R)\langle x, y\rangle$ a skew PBW extension of endomorphism type over $R$, then for any $r\in R$ we have the relations $xr=\sigma_1(r)x$, $yr=\sigma_2(r)y$, and $yx = d_{1,2}xy + r_0 + r_1x + r_2y$, for some elements $d_{1,2}, r_0, r_1$ and $r_2$ belong to $R$. When we compare the defining relations of both algebraic structures, it is clear which one of them is more general.

Now, if we have the iterated Ore extension $R[x;\sigma_x][y;\sigma_y][z;\sigma_z]$, then for any $r\in R$, $xr = \sigma_x(r)x$, $yr = \sigma_y(r)y$, $zr = \sigma_z(r)z$, $yx = \sigma_y(x)y$, $zx = \sigma_z(x)z$, $zy = \sigma_z(y)z$. For the skew PBW extension of endomorphism type $\sigma(R)\langle x, y, z\rangle$, we have the relations given by  $xr=\sigma_1(r)x$, $yr=\sigma_2(r)y$, $zr = \sigma_3(r)z$, $yx = d_{1,2}xy + r_0 + r_1x + r_2y + r_3z$, $zx = d_{1,3}xz + r_0' + r_1'x + r_2'y + r_3'z$, and $zy = d_{2,3}yz + r_0'' + r_1''x + r_2''y + r_3''z$, for some elements $d_{1,2}, d_{1,3}, d_{2,3}, r_0, r_0', r_0'', r_1, r_1', r_1'', r_2, r_2', r_2'', r_3$, $r_3', r_3''$ of $R$. Again, one can see the generality of the skew PBW extensions. Of course, as the number of indeterminates increases, the differences between both algebraic structures are more remarkable.
%\item The class of skew polynomial rings of injective type is strictly contained in the family of skew PBW extensions. For instance, if $R$ is a commutative ring and $\mathfrak{g}$ is a finite dimensional Lie algebra over $R$ with basis $\left \{x_1,\dots,x_n \right \}$, then the {\it universal enveloping algebra} of $\mathfrak{g}$, denoted $U(\mathfrak{g})$, is a PBW extension over $R$, since, for all $r \in R$ and every $x_i,x_j$, $x_ir - rx_i = 0 \in R$, and $x_ix_j - x_jx_i = [x_i,x_j] \in \mathfrak{g}$ where $[x_i,x_j] = R x_i +\cdots + R x_n$, for all $1 \leq i,j \leq n$. 
%\item PBW extensions introduced by Bell and Goodearl are not more general than skew polynomial rings of injective type. \textcolor{blue}{Complete}
\end{itemize}
\end{remark}
%%%%%%%%%%%%%%%%%%%%%%%%%%%%%%%%%%%%%%%%%%%%
\begin{example}
A great variety of algebras can be expressed as skew PBW extensions. Enveloping algebras of finite dimensional Lie algebras, families of differential operators rings, Weyl algebras, Ore extensions of injective type, some types of Auslander-Gorenstein rings, some skew Calabi Yau algebras, Artin-Schelter regular algebras, Koszul algebras, examples of quantum polynomials, some quantum universal enveloping algebras, and many other algebras of great interest in noncommutative algebraic geometry and noncommutative differential geometry illustrate the generality of skew PBW extensions. Details of these examples and others can be found in \cite{Fajardoetal2020, GomezSuarez2019,Suarez2017}.
\end{example}
%%%%%%%%%%%%%%%%%%%%%%%%%%%%%%%%%%%%%%%%%%%%
The next proposition is very useful when one need to make some computations with elements of skew PBW extensions.
%%%%%%%%%%%%%%%%%%%%%%%%%%%%%%%
\begin{proposition}(\cite[Theorem 7]{LezamaGallego}) \label{coefficientes}
If $A$ is a polynomial ring with coefficients in $R$ with respect to the set of indeterminates $\{x_1,\dots,x_n\}$, then $A$ is a skew PBW extension over $R$ if and only if the following conditions hold:
\begin{enumerate}
\item[\rm (1)]for each $x^{\alpha}\in {\rm Mon}(A)$ and every $0\neq r\in R$, there exist unique elements $r_{\alpha}:=\sigma^{\alpha}(r)\in R\ \backslash\ \{0\}$, $p_{\alpha ,r}\in A$, such that $x^{\alpha}r=r_{\alpha}x^{\alpha}+p_{\alpha, r}$, where $p_{\alpha ,r}=0$, or $\deg(p_{\alpha ,r})<|\alpha|$ if $p_{\alpha , r}\neq 0$. If $r$ is left invertible, so is $r_\alpha$.
\item[\rm (2)]For each $x^{\alpha},x^{\beta}\in {\rm Mon}(A)$, there exist unique elements $d_{\alpha,\beta}\in R$ and $p_{\alpha,\beta}\in A$ such that $x^{\alpha}x^{\beta} = d_{\alpha,\beta}x^{\alpha+\beta}+p_{\alpha,\beta}$, where $d_{\alpha,\beta}$ is left invertible, $p_{\alpha,\beta}=0$, or $\deg(p_{\alpha,\beta})<|\alpha+\beta|$ if
$p_{\alpha,\beta}\neq 0$.
\end{enumerate}
\end{proposition}
%%%%%%%%%%%%%%%%%%%%%%%%%%%%%%%%%%%%%%
We need to establish a criterion which allows us to extend the family $\Sigma$ of injective endomorphisms and the family of $\Sigma$-derivations $\Delta$ of the ring $R$, to any skew PBW extension $A$ over $R$ (c.f. Artamonov \cite{Artamonov2015} and Venegas \cite{Venegas2015} who presented a study of derivations and automorphisms of skew PBW extensions, respectively). With this aim, for the next result consider the pair $(\Sigma, \Delta)$, the system of endomorphisms and $\Sigma$-derivations of $R$ with respect to $A$.
%%%%%%%%%%%%%%%%%%%%%%%%%%%%%%%%%%%%%%
\begin{proposition}(\cite[Theorem 5.1]{ReyesSuarez2018-3}) \label{prop-indicsigma}
Let $A = \sigma(R)\langle x_1,\dotsc, x_n\rangle$ be a skew PBW extension over a ring $R$. Suppose that $\sigma_i\delta_j=\delta_j\sigma_i,\ \delta_i\delta_j=\delta_j\delta_i$, and $\delta_k(d_{i,j}) =
\delta_k(r_l^{(i,j)}) = 0$, for $1\le i, j, k, l\le n$, where $d_{i,j}$
and $r_l^{(i,j)}$ are the elements established in Definition
\ref{def.skewpbwextensions}. Consider an element  $f=a_0 + a_1X_1+\dotsb + a_mX_m\in A$. If $\overline{\sigma_{k}}:A\to A$ and $\overline{\delta_k}:A\to A$ are the functions given by $\overline{\sigma_{k}}(f):=\sigma_k(a_0)+\sigma_k(a_1)X_1 + \dotsb + \sigma_k(a_m)X_m$ and $\overline{\delta_k}(f):=\delta_k(a_0) +
\delta_k(a_1)X_1 + \dotsb + \delta_k(a_m)X_m$, respectively, and $\overline{\sigma_k}(r):=\sigma_k(r)$, for every $1\le k\le n$ and $r\in R$, then $\overline{\sigma_k}$ is an injective endomorphism of $A$ and $\overline{\delta_k}$ is a $\overline{\sigma_k}$-derivation of $A$, for each $k$.
\end{proposition}
%%%%%%%%%%%%%%%%%%%%%%%%%%%%%%%%%%%%%%%%%%%%
As we said in Section \ref{sect.prelim}, Krempa \cite{Krempa1996} introduced the notion of rigidness as a generalization of reduced rings. Now, if $R$ is a ring and $\Sigma=\{\sigma_1,\dots,\sigma_n\}$ is a family of endomorphisms of $R$, then $\Sigma$ is called a \emph{rigid endomorphisms family} if $a\sigma^{\alpha}(a)= 0$ implies $a = 0$, where $a\in R$ and $\alpha\in \mathbb{N}^n$. $R$ is called $\Sigma$-\emph{rigid} if there exists a rigid endomorphisms family $\Sigma$ of $R$ \cite{ReyesSuarez2018-3}, Definition 3.1. Following Annin \cite{Annin2004} or Hashemi and Moussavi \cite{HashemiMoussavi2005}, if $R$ is a ring, $\sigma$ is an endomorphism of $R$, and $\delta$ is a $\sigma$-derivation of $R$, then $R$ is said to be $\sigma$-{\em compatible} if for each $a, b\in R$, $ab = 0$ if and only if $a\sigma(b)=0$; in this case is said that $\sigma$ is {\em  compatible} (necessarily, the endomorphism $\sigma$ is injective). $R$ is called $\delta$-{\em compatible} if for each $a, b\in R$, $ab = 0$ implies $a\delta(b)=0$. If $R$ is both $\sigma$-compatible and $\delta$-compatible, $R$ is called ($\sigma,\delta$)-{\em compatible}. The corresponding notion of compatibility, and an even weaker notion, have been formulated for skew PBW extensions as the following definition shows.
%%%%%%%%%%%%%%%%%%%%%%%%%%%%%%%%%%%
\begin{definition}
Consider a ring $R$ with finite families of endomorphisms $\Sigma$ and $\Sigma$-derivations $\Delta$, respectively. 
\begin{enumerate}
    \item [\rm (1)] (\cite[Definition 3.1]{HashemiKhalilAlhevaz2017}; \cite[Definition 3.2]{ReyesSuarez2018RUMA}) $R$ is said to be {\it $\Sigma$-compatible} if for each $a, b \in R$, $a\sigma^{\alpha}(b) = 0$ if and only if $ab = 0$, where $\alpha \in \mathbb{N}^n$. In this case, we say that $\Sigma$ is a {\it compatible family} of endomorphisms of $R$. $R$ is said to be {\it $\Delta$-compatible} if for each $a, b \in R$, it follows that $ab = 0$ implies $a\delta^{\beta}(b)=0$, where $\beta \in \mathbb{N}^n$. If $R$ is both $\Sigma$-compatible and $\Delta$-compatible, then $R$ is called {\it  $(\Sigma, \Delta)$-compatible}.
    \item [\rm (2)] \cite[Definition 4.1]{ReyesSuarez2019-2} $R$ is said to be {\it weak $\Sigma$-compatible} if for each $a, b \in R$, $a\sigma^{\alpha}(b)\in N(R)$ if and only if $ab \in N(R)$, where $\alpha \in \mathbb{N}^n$. $R$ is said to be {\it weak $\Delta$-compatible} if for each $a, b \in R$, $ab \in N(R)$ implies $a\delta^{\beta}(b)\in N(R)$, where $\beta \in \mathbb{N}^n$. If $R$ is both weak $\Sigma$-compatible and weak $\Delta$-compatible, then $R$ is called {\it weak $(\Sigma, \Delta)$-compatible}.
\end{enumerate}
\end{definition}
%%%%%%%%%%%%%%%%%%%%%%%%%%%%%%%%%%
Ring-theoretical properties of skew PBW extensions over compatible rings have been investigated by the authors \cite{ReyesSuarez2019-1,ReyesSuarez2019Radicals}.
%%%%%%%%%%%%%%%%%%%%%%%%%%%%%%%%%%
\begin{remark}\label{remark-comsiiSigma}
It is straightforward to prove that if $\Sigma$ is a finite family of endomorphisms of a ring $R$, then $R$ is $\Sigma$-compatible if and only if $R$ is $\sigma_i$-compatible, for every $1\leq i\leq n$.
\end{remark}
%%%%%%%%%%%%%%%%%%%%%%%%%%%%%%%%%%%%%%%%%%%
\begin{lemma}{\rm (}\cite[Lemma 3.3]{HashemiKhalilAlhevaz2017}; \cite[Proposition 3.8]{ReyesSuarez2018RUMA}{\rm )}\label{colosss}
Let $R$ be a $(\Sigma, \Delta)$-compatible ring. For every $a, b \in R$, we have the following:
\begin{enumerate}
\item [\rm (1)] If $ab=0$, then $a\sigma^{\theta}(b) = \sigma^{\theta}(a)b=0$, where $\theta\in \mathbb{N}^{n}$.
\item [\rm (2)] If $\sigma^{\beta}(a)b=0$, for some $\beta\in \mathbb{N}^{n}$, then $ab=0$.
\item [\rm (3)] If $ab=0$, then $\sigma^{\theta}(a)\delta^{\beta}(b)= \delta^{\beta}(a)\sigma^{\theta}(b) = 0$, where $\theta, \beta\in \mathbb{N}^{n}$.
\end{enumerate}
\end{lemma}
%%%%%%%%%%%%%%%%%%%%%%%%%%%%%%%%%%%%%%%%%%%
\begin{lemma}\label{lemma-NilSigma-comp}
If $R$ is a $\Sigma$-compatible ring, $a,b\in R$ and $\alpha, \beta\in \mathbb{N}^n$, then the following assertions are equivalent:
\begin{enumerate}
\item[\rm (1)] $ab\in N(R)$.
\item[\rm (2)] $a\sigma^{\alpha}(b)\in N(R)$.
\item[\rm (3)] $\sigma^{\alpha}(b)a\in N(R)$.
\item[\rm (4)] $\sigma^{\alpha}(a)b\in N(R)$.
\item[\rm (5)] $b\sigma^{\alpha}(a)\in N(R)$.
\item[\rm (6)] $\sigma^{\alpha}(a)\sigma^{\beta}(b)\in
N(R)$.
\end{enumerate}
\end{lemma}
\begin{proof}
Let $R$ be a $\Sigma$-compatible ring, $a, b\in R$ and $\alpha,\beta\in \mathbb{N}^n$.
 
(1) $\Leftrightarrow$ (2) By definition, $ab\in N(R) \Leftrightarrow (ab)^k = 0$, for some positive integer $k$. If $k=1$, $ab =0$ if and only if
$a\sigma^{\alpha}(b)$. If $k=2$, $0= abab\Leftrightarrow 0=a\sigma^{\alpha}(bab)=a\sigma^{\alpha}(b)\sigma^{\alpha}(ab)\Leftrightarrow
a\sigma^{\alpha}(b)ab=0\Leftrightarrow
0=a\sigma^{\alpha}(b)a\sigma^{\alpha}(b)=(a\sigma^{\alpha}(b))^2$, where all equivalences are due to the $\Sigma$-compatibility of $R$. Continuing with this process, $(ab)^k=0\Leftrightarrow (a\sigma^{\alpha}(b))^k=0$ Thus, $ab\in N(R)$ if and only if $a\sigma^{\alpha}(b)\in N(R)$.

(2) $\Leftrightarrow$ (3) It is clear.

(1) $\Leftrightarrow$ (4) Again, since $ab\in N(R) \Leftrightarrow (ab)^k = 0$, for some positive integer $k$, if $k=1$, Lemma \ref{colosss} implies that $ab =0 \Leftrightarrow \sigma^{\alpha}(a)b=0$. If $k=2$, $0= abab\Leftrightarrow 0 = \sigma^{\alpha}(a)bab\Leftrightarrow
\sigma^{\alpha}(a)b\sigma^{\alpha}(ab)=\sigma^{\alpha}(a)b\sigma^{\alpha}(a)\sigma^{\alpha}(b)$,
$\Leftrightarrow \sigma^{\alpha}(a)b\sigma^{\alpha}(a)b=0$
$\Leftrightarrow (\sigma^{\alpha}(a)b)^2=0$, where the first equivalence is due to Lemma \ref{colosss}, and the others equivalences are due to the $\Sigma$-compatibility of $R$. Continuing in this way, we can see that $(ab)^k=0\Leftrightarrow (\sigma^{\alpha}(a)b)^k=0$. Hence $ab\in
N(R) \Leftrightarrow \sigma^{\alpha}(a)b\in N(R)$.

(4) $\Leftrightarrow$ (5) It is immediate.

(5) $\Leftrightarrow$ (6) It follows from (1) $\Leftrightarrow$ (2) by replacing $a$ by
$\sigma^{\alpha}(a)$ and $\sigma^{\alpha}(b)$ by
$\sigma^{\beta}(b)$.
\end{proof}
%%%%%%%%%%%%%%%%%%%%%%%%%%%%%%%%%%%%%%%%%%%
To finish this section, we recall the notion of Armendariz ring in the setting of commutative and noncommutative rings. Following Rege and Chhawchharia \cite{RegeChhawchharia1997}, a ring $R$ is called {\em Armendariz} if for elements $f(x) = \sum_{i=0}^{m} a_ix^{i},\ g(x) = \sum_{j=0}^{l} b_jx^{j}\in R[x]$ (where $R[x]$ is the commutative polynomial ring with an indeterminate $x$ over $R$), $f(x)g(x) = 0$ implies $a_ib_j = 0$, for all $i, j$. The importance of this notion lies in its natural and useful role in understanding the relation between the annihilators of the ring $R$ and the annihilators of the polynomial ring $R[x]$. For example, Armendariz \cite[Lemma 1]{Armendariz1974}, showed that a reduced ring always satisfies this condition. For skew polynomial rings, the notion of Armendariz has been also studied by several authors (see \cite{AndersonCamillo1998,Armendariz1974,HongKimKwak2003,Huhetal2002,KimLee2000,LeeWong2003,RegeChhawchharia1997}). Hirano \cite{Hirano2002} generalized this notion in the following way: a ring $R$ is called {\em quasi-Armendariz} if for $f(x) = \sum_{i=0}^{m} a_ix^{i},\ g(x) = \sum_{j=0}^{l} b_jx^{j}$ in $R[x]$, $f(x)R[x]g(x) = 0$ implies $a_iRb_j = 0$, for all $i, j$. It is well known that reduced rings are Armendariz, and Armendariz rings are quasi-Armendariz, but the converse are not true in general \cite[Proposition 2.1]{Rege1997}.

\vspace{0.2cm}

In the noncommutative setting, Hashemi et al. \cite{HashemiMoussavi2005} investigated a generalization of $\sigma$-rigid rings by introducing the (SQA1) condition which is a version of quasi-Armendariz rings for skew polynomial rings. Let $R[x;\sigma,\delta]$ be an Ore extension over $R$. $R$ is said to satisfy the ({\em SQA1}) {\em condition} if whenever $f(x)R[x; \sigma, \delta]g(x) = 0$ for $f(x) = \sum_{i=0}^m a_ix^i$ and $g(x) = \sum_{j=0}^m b_jx^j \in R[x; \sigma, \delta]$, then $a_iRb_j = 0$ for all $i, j$. Later, Reyes and Su\'arez \cite{ReyesSuarez2018RUMA} considered this condition for skew PBW extensions. For $A = \sigma(R)\langle x_1,\dotsc, x_n\rangle$ a skew PBW extension over $R$, $R$ is said to satisfy the ({\em SQA1}) {\em condition} if whenever $fAg=0$ for $f=a_0+a_1X_1+\dotsb + a_mX_m$ and $g=b_0 + b_1Y_1 + \dotsb + b_tY_t$ elements of $A$, then $a_iRb_j = 0$, for every $i, j$. Relationships between the notions of compatibility, (SQA1), and Armendariz rings, in the context of skew PBW extensions, can be found in \cite{ReyesSuarez2018-3,ReyesSuarez2019Radicals}. All of them will be useful in the next sections.
%%%%%%%%%%%%%%%%%%%%%%%%%%%%%%%
%%%%%%%%%%%%%%%%%%%%%%%%%%%%%%%
\section{\texorpdfstring{$\Sigma -$}\ skew reflexive rings}\label{sect-Ore ext}
We present the original results of the paper on $\Sigma$-skew reflexive rings. We start with the following definition which establishes the generalizations of skew reflexive endomorphism and $\sigma$-skew reflexive ring mentioned in the Introduction.
%%%%%%%%%%%%%%%%%%%%%%%%%%%%%%%
\begin{definition}\label{def-skewReflex}
Let $\Sigma=\{\sigma_1,\dots, \sigma_n\}$ be a finite family of endomorphisms of a ring $R$. $\Sigma$ is called \emph{right} (resp., \emph{left}) {\em skew reflexive} if for $a, b\in R$, $aRb = 0$ implies $bR\sigma^{\alpha}(a) = 0$ (resp., $\sigma^{\alpha}(b)Ra = 0$), where $\alpha\in \mathbb{N}^n\ \backslash\ \{0\}$. $R$ is called right (resp., {\em left}) $\Sigma$-{\em skew reflexive} if there exists a right (resp., left) skew reflexive family of endomorphisms $\Sigma$ of $R$. If $R$ is both right and left $\Sigma$-skew reflexive, then $R$ is said to be $\Sigma$-{\em skew reflexive}.
\end{definition}
%%%%%%%%%%%%%%%%%%%%%%%%%%%%%
Reflexive rings are $\Sigma$-skew reflexive. Since reduced rings and reversible rings are reflexive, then both rings are $\Sigma$-skew reflexive. 
Right (resp., left) $\sigma$-skew reflexive rings are right (resp., left) $\Sigma$-skew reflexive.
%%%%%%%%%%%%%%%%%%%%%%%%%%%%%
\begin{example}\label{exampleS2} 
Given a ring $R$ and $M$ an $(R,R)$-bimodule, \textit{the trival extension of R by M} is the ring $T(R,M):= R\oplus M$ with the usual addition and the multiplication defined as $(r_1,m_1)(r_2,m_2):=(r_1r_2, r_1m_2 + m_1r_2)$, for all $r_1,r_2 \in R$ and $m_1,m_2 \in M$. Note that $T(R,M)$ is isomorphic to the matrix ring (with the usual matrix operations) of the form $\bigl(\begin{smallmatrix}r & m\\ 0 & r \end{smallmatrix}\bigr)$, where $r\in R$ and $m \in M$. In particular, we call $S_2(\mathbb{Z})$ the ring of matrices isomorphic to $T(\mathbb{Z},\mathbb{Z})$.

Consider the ring $S_2(\mathbb{Z})$ given by
\begin{center}
    $S_2(\mathbb{Z})= \left \{ \begin{pmatrix}a & b\\ 0 & a \end{pmatrix} \mid a,b \in \mathbb{Z} \right \}$.
\end{center}
Let $\sigma_1={\rm id}_{S_2(\mathbb{Z})}$ be the identity endomorphism of $S_2(\mathbb{Z})$, and let $\sigma_2$ and $\sigma_3$ be the two endomorphisms of $S_2(\mathbb{Z})$ defined by
\begin{center}
    $\sigma_2\left ( \begin{pmatrix}a & b\\ 0 & a \end{pmatrix} \right )= \begin{pmatrix}a & -b\\0 & a\end{pmatrix}, \ \ \ \ \ \ \sigma_3\left ( \begin{pmatrix}a & b\\ 0 & a \end{pmatrix} \right )= \begin{pmatrix}a & 0\\0 & a\end{pmatrix}$.
\end{center}
Consider $ARB = 0$, for all $R \in S_2(\mathbb{Z})$, where 
\begin{center}
    $A=\begin{pmatrix} a & a'\\ 0& a\end{pmatrix},\ \ B=\begin{pmatrix} b & b'\\ 0 & b \end{pmatrix}\ \ \text{and}\ \ R=\begin{pmatrix} r & r'\\ 0 & r \end{pmatrix}$.
\end{center}
We have $arb = 0$ and $arb' + ar'b + bra' = 0$, for every $r \in \mathbb{Z}$, whence $a = 0$ or $b = 0$. If $a = 0$, then $bra' = 0$, for every $r \in \mathbb{Z}$, hence $b = 0$ or $a' = 0$. If $b = 0$, then $arb' = 0$, for each $r \in \mathbb{Z}$, hence $a = 0$ or $b' = 0$. First, we consider the case $BR\sigma_1(A)=BRA$:
\begin{center}
    $BR\sigma_1(A)=\begin{pmatrix} bra & bra' + br'a + b'ra\\ 0& bra\end{pmatrix}$.
\end{center}
From the previous observation, it follows that $BR\sigma_1(A)=0$. Now, consider the case $BR\sigma_2(A)$. A calculation shows us that
\begin{center}
    $BR\sigma_2(A)= \begin{pmatrix} bra & -bra' + br'a + b'ra\\ 0& bra\end{pmatrix}$.
\end{center}
Again, making use of the initial observation, we can see that $BR\sigma_2(A)=0$. $S_2(\mathbb{Z})$ is right $\sigma_2$-skew reflexive \cite[Proposition 2.11]{Kwaketal2014}. Finally, we present the case $BR\sigma_3(A)$:
\begin{center}
    $BR\sigma_3(A)= \begin{pmatrix} bra &  br'a + b'ra\\ 0& bra\end{pmatrix}$.
\end{center}
We obtain that $bra=0$, since $BRA=0$. If $a=0$, then $br'a + b'ra=0$, for every $r,r' \in \mathbb{Z}$. If $b=0$, it follows that $b'ra + br'a + bra' = br'a + b'ra = 0$, for each $r,r'\in \mathbb{Z}$. This implies that $BR\sigma_3(A)=0$. We can also observe that $\sigma_2 \circ \sigma_3=\sigma_3 \circ \sigma_2=\sigma_3$, which implies that $S_2(\mathbb{Z})$ is right $\Sigma$-skew reflexive with $\Sigma=\left \{ \sigma_1,\sigma_2,\sigma_3 \right \}$.

Since $\sigma_1$ is the identity endomorphism, then we have that
$S_2(\mathbb{Z})$ is a reflexive ring. Now, $S_2(\mathbb{Z})$ is not
$\sigma_3$-compatible since for
\begin{center}
    $A=\begin{pmatrix} 1 & 1\\ 0& 1\end{pmatrix},\ \ B=\begin{pmatrix} 0 & 1\\ 0 & 0 \end{pmatrix},$
\end{center}
$A\sigma_3(B)=0$ but $AB=\begin{pmatrix} 0 & 1\\ 0&
0\end{pmatrix}\neq 0.$ According to Remark \ref{remark-comsiiSigma}, $S_2(\mathbb{Z})$ is not $\Sigma$-compatible.
\end{example}
%%%%%%%%%%%%%%%%%%%%%%%%%%%%%%%%
The following result shows that under conditions of compatibility, the composition of left (resp. right) skew reflexive endomorphisms is left (resp. right) skew reflexive endomorphism.
%%%%%%%%%%%%%%%%%%%%%%%%%%%%%%
\begin{proposition}\label{prop-composkewReflex}
Composition of compatible left (resp. right)  skew reflexive endomorphisms is left (resp. right) skew reflexive.
\end{proposition}
\begin{proof}
Let $\sigma_1$ and $\sigma_2$ be compatible endomorphisms of a ring $R$ and consider $a,b\in R$. Suppose that $aRb=0$, i.e. $arb=0$, for every $r\in R$. If $\sigma_1$ and $\sigma_2$ are left skew reflexive, then $\sigma_1(b)ra = \sigma_2(b)ra = 0$. Hence $\sigma_1(\sigma_2(b)ra)=
\sigma_1(\sigma_2(b))\sigma_1(ra)=\sigma_1(\sigma_2(b))ra=0$, since $\sigma_1$ is compatible. So, $\sigma_1\circ \sigma_2$ is left skew reflexive. Now, if $\sigma_1$ and $\sigma_2$ are right skew
reflexive endomorphisms, then $arb=0$ implies $br\sigma_1(a) = br\sigma_2(a) = 0$, whence $\sigma_1(br\sigma_2(a))=\sigma_1(br)\sigma_1(\sigma_2(a))=0$. Since $\sigma_1$ is compatible,  $\sigma_1(br)\sigma_1(\sigma_2(a))=0$ implies $br\sigma_1(\sigma_2(a))=0$, and therefore $\sigma_1\circ \sigma_2$ is a right skew reflexive endomorphism.
\end{proof}
%%%%%%%%%%%%%%%%%%%%%%%%%%%%%%%%%%%%%
From Proposition \ref{prop-composkewReflex} we deduce the following results.
%%%%%%%%%%%%%%%%%%%%%%%%%%%%%%%%%%%%%
\begin{corollary}\label{cor-composkewReflexive}
Composition of a finite family of compatible left (resp. right) skew reflexive  endomorphisms is a left (resp. right) skew reflexive endomorphism. If $R$ is $\sigma$-compatible  then $R$ is right (resp., left) $\sigma$-skew reflexive if and only if it is right (resp., left) $\sigma^{k}$-skew reflexive, for all $k\geq 1$. Thus, if $R$ is $\sigma$-compatible then $R$ is $\sigma$-skew reflexive if and only if it is $\sigma^{k}$-skew reflexive, for all $k\geq 1$.
\end{corollary}
%%%%%%%%%%%%%%%%%%%%%%%%%%%%%%%%%%%%%
\begin{proposition}\label{prop-Sigmarefiffsigma}
Let $\Sigma=\{\sigma_1,\dots, \sigma_n\}$ be a compatible family of endomorphisms of $R$. Then  $\Sigma$ is a right (left) skew
reflexive family if and only if each $\sigma_k\in \Sigma$ is a right
(left) skew reflexive endomorphism.
\end{proposition}
\begin{proof}
Let $\sigma_k\in \Sigma$. Note that for
$\alpha=(\alpha_1,\dots,\alpha_n)\in \mathbb{N}^n$, $\sigma^{\alpha}= \sigma_k$, for $\alpha_i=0$ if $i\neq k$ and $\alpha_k=1$. If $\Sigma$ is a right (resp., left) skew reflexive family of endomorphisms,  then for $a, b\in R$, $aRb = 0$ implies $bR\sigma^{\alpha}(a) = 0= bR\sigma_k(a)$ (resp., $\sigma^{\alpha}(b)Ra=\sigma_k(b)Ra=0$). So, $\sigma_k$ is a right (resp., left) skew reflexive endomorphism. On the other hand, if each $\sigma_k\in\Sigma$ is a compatible right (left) skew reflexive endomorphism  then by Corollary \ref{cor-composkewReflexive} we have that $\sigma^{\alpha}$ is right (resp., left) skew reflexive, for every $\alpha\in \mathbb{N}^n$. Then for $a, b\in R$, $aRb = 0$ implies $bR\sigma^{\alpha}(a)$ (resp., $\sigma^{\alpha}(b)Ra=0$), for  every $\alpha\in \mathbb{N}^n$. Therefore $\Sigma$ is a right (left) skew reflexive family.
\end{proof}
%%%%%%%%%%%%%%%%%%%%%%%%%%
\begin{proposition}\label{prop-SigmaRNPimlArefl}
Let $A = \sigma(R)\langle x_1,\dotsc, x_n\rangle$ be a quasi-commutative skew PBW extension over a $\Sigma$-skew reflexive ring $R$, that is, $A$ is an iterated skew polynomial ring of injective endomorphism type over $R$. If $R$  satisfies the (SQA1) condition, then $A$ is reflexive.
\end{proposition}
\begin{proof}
Let $\Sigma=\{\sigma_1,\dots,\sigma_n\}$ be the family of injective endomorphisms of $R$ as in Proposition \ref{sigmadefinition}. Consider $f=a_0+a_1X_1+\dotsb + a_mX_m$ and $g=b_0 + b_1Y_1 + \dotsb + b_tY_t$ elements of $A$. If $fAg=0$, by the (SQA1) condition of $R$ we obtain $a_iRb_j = 0$, for each $i, j$. Since $R$ is $\Sigma$-skew reflexive, $b_jR\sigma^{\alpha}(a_i) = \sigma^{\alpha}(b_j)Ra_i = 0$, where $\alpha\in \mathbb{N}^n\ \backslash\ \{0\}$, and each $1\leq i\leq m$, $1\leq j\leq t$. As $A$ is quasi-commutative, by Proposition \ref{coefficientes} we know that for every $x^{\alpha}\in {\rm Mon}(A)$ and every $0\neq
r\in R$, there exists an element $\sigma^{\alpha}(r)\in R\ \backslash\ \{0\}$ such that $x^{\alpha}r=\sigma^{\alpha}(r)x^{\alpha}$, and for every $x^{\alpha}, x^{\beta}\in {\rm Mon}
(A)$, there exists an element $d_{\alpha,\beta}\in R$ such that $x^{\alpha}x^{\beta} = d_{\alpha,\beta}x^{\alpha+\beta}$, where $d_{\alpha,\beta}$ is left invertible. Thus, applying these commutation rules to the product  $(b_0 + b_1Y_1 + \dotsb +
b_tY_t)(c_0+c_1Z_1+\dotsb + c_lZ_l)(a_0+a_1X_1+\dotsb + a_mX_m)$,
where $c_0+c_1Z_1+\dotsb + c_lZ_l$ is an  arbitrary element of $A$, and taking into account that $b_jR\sigma^{\alpha}(a_i) = \sigma^{\alpha}(b_j)Ra_i = 0$, for $\alpha\in \mathbb{N}^n\
\backslash\ \{0\}$, and every $i$, $j$, then $gAf=0$, that is, $A$ is reflexive.
\end{proof}
%%%%%%%%%%%%%%%%%%%%%%%%%%%%%%%%%%%%%%%%%%%%
From Proposition \ref{prop-SigmaRNPimlArefl} we obtain that if $R$ is a reflexive ring, then $R[x]$ is right idempotent reflexive \cite[Theorem 3.13 (1)]{KwakLee2012}.
%%%%%%%%%%%%%%%%%%%%%%%%%%%%%%%%%%%%
\begin{proposition}\label{prop-semiprimplSigmaskew}
$\Sigma$-compatible semiprime rings are $\Sigma$-skew reflexive.
\end{proposition}
\begin{proof}
Suppose that $R$ is semiprime and $\Sigma$-compatible for $\Sigma=\{\sigma_1,\dotsc, \sigma_n\}$ a finite family of endomorphisms
of $R$. Let $a,b\in R$ such that $aRb=0$. By the $\Sigma$-compatibility of $R$, $aR\sigma^{\alpha}(b)=0$, where $\alpha\in \mathbb{N}^n$, whence $\sigma^{\alpha}(b)RaR\sigma^{\alpha}(b)Ra=0$. Since $R$ is semiprime, $\sigma^{\alpha}(b)Ra=0$, and thus $R$ is left $\Sigma$-skew reflexive. By Remark \ref{remark-comsiiSigma}, $aRb=0$ implies $\sigma^{\alpha}(a)Rb=0$, and so $bR\sigma^{\alpha}(a)RbR\sigma^{\alpha}(a)=0$. The semiprimeness of $R$ implies  $bR\sigma^{\alpha}(a)=0$, and therefore $R$ is right $\Sigma$-skew reflexive.
\end{proof}
%%%%%%%%%%%%%%%%%%%%%%%%%%%%%%%%%%%%%%%%%%%%
Propositions \ref{prop-SigmaRNPimlArefl} and
\ref{prop-semiprimplSigmaskew} guarantee that quasi-commutative skew PBW extensions over $\Sigma$-compatible semiprime rings satisfying the (SQA1) condition are reflexive.
%%%%%%%%%%%%%%%%%%%%%%%%%%%%%%%%%%%%%%%%%%%%
\begin{remark}
Following Birkenmeier et al. \cite{Birkenmeieretal2001}, a ring $R$ is called {\em right} ({\em left}) {\em principally quasi-Baer} (or simply, {\em right} ({\em left}) {\em p.q.-Baer}) ring if the right annihilator of a principal right ideal of $R$ is generated by an idempotent. $R$ is said to be a {\em p.q.-Baer} if it is both right and left p.q.-Baer. Having in mind that Baer properties have been studied in the setting of skew PBW extensions \cite{ReyesSuarez2018-3,ReyesSuarez2018RUMA}, if $A$ is a skew PBW extension right p.q.-Baer ring, then $A$ is a semiprime ring if and only if $A$ is a reflexive ring \cite[Proposition 3.15]{KwakLee2012}.
\end{remark}
%%%%%%%%%%%%%%%%%%%%%%%%%%%%%%%%%%%%%%%%%%%%
%%%%%%%%%%%%%%%%%%%%%%%%%%%%%%%%%%%%%%%%%%%%
\section{Skew PBW extension over \texorpdfstring{$\Sigma -$}\ skew RNP rings}\label{sect-NocSigma}
Motivated by the notion put forward by Bhattacharjee \cite[Definition 1.1]{Bhattacharjee2020},  we introduce the following definition.
%%%%%%%%%%%%%%%%%%%%%%%%%%%%%%%%%%%%%%%
\begin{definition}
Let $\Sigma=\{\sigma_1,\dots, \sigma_n\}$ be a finite family of endomorphisms of a ring $R$. $\Sigma$ is called a \emph{right} (resp., \emph{left}) {\em skew RNP family} if for $a, b\in N(R)$, $aRb = 0$ implies $bR\sigma^{\alpha}(a) = 0$ (resp., $\sigma^{\alpha}(b)Ra = 0$), for  every nonzero element $\alpha\in \mathbb{N}^n$. $R$ is said to be right (resp., left) $\Sigma$-{\em skew RNP} if there exists a right (resp., left) skew RNP family of endomorphism $\Sigma$ of $R$. If $R$ is both right and left $\Sigma$-skew RNP, then we say that $R$ is $\Sigma$-{\em skew RNP}.
\end{definition}
%%%%%%%%%%%%%%%%%%%%%%%%%%%%%%%%%%%%%%%%%%%%
It is straightforward to show that right (resp., left) $\Sigma$-skew reflexive rings are right (resp., left) $\Sigma$-skew RNP, reduced rings are $\Sigma$-skew RNP, for any finite family of endomorphisms $\Sigma$, and RNP rings are $\Sigma$-skew RNP considering $\Sigma$ as the singleton consisting of the identity function of $R$.
%%%%%%%%%%%%%%%%%%%%%%%%%%%%%%%%%%%%%%%%%
\begin{remark}\label{rem-SigmaN(R)contN(R)} Note that for $\Sigma=\{\sigma_1, \dots, \sigma_n\}$, $\sigma^{\alpha}(N(R))\subseteq N(R)$, for each $\alpha\in \mathbb{N}^n$. Indeed,  if $r\in N(R)$ then $r^m=0$, for some $m\geq 1$. So, $\sigma^{\alpha}(r)\in \sigma^{\alpha}(N(R))$ implies that $\sigma^{\alpha}(r^n) = (\sigma^{\alpha}(r))^n = 0$, whence $\sigma^{\alpha}(r)\in N(R)$.
\end{remark}
%%%%%%%%%%%%%%%%%%%%%%%%%%%%%%%%%%%%%%%
\begin{example}
Considering Example \ref{exampleS2}. It is clear that $S_2(\mathbb{Z})$ is a right $\Sigma$-skew RNP ring with $\Sigma= \left \{\sigma_1,\sigma_2,\sigma_3 \right \}$. Furthermore, we can see that $N(S_2(\mathbb{Z})) = \left \{\begin{pmatrix}0 & b\\ 0 & 0 \end{pmatrix} \mid b \in \mathbb{Z} \right \}$.
\end{example}
%%%%%%%%%%%%%%%%%%%%%%%%%%%%%%%%%%%%%%%
\begin{example}\label{ejemploraro}
Let $\mathbb{F}_4=\left \{0,1,a,a^2 \right \}$ be the finite field of four elements. Consider the ring of polynomials $\mathbb{F}_4[z]$ and let $R=\frac{\mathbb{F}_4[z]}{\left \langle z^2 \right \rangle}$. For simplicity, we identify the elements of $\mathbb{F}_4[z]$ with their images in $R$. Let $\Sigma=\left \{\sigma_{i,j} \right \}$ be the family of homomorphisms of $R$ defined by $\sigma_{i,j}(a)=a^i$ and $\sigma_{i,j}(z)=a^jz$, for all $a \in \mathbb{Z}_3$ with $1 \leq i \leq 2$ and $0\leq j \leq 2$. Let us consider the skew PBW extension defined as $A=\sigma(R)\left \langle x_{1,0},x_{1,1},x_{1,2},x_{2,0},x_{2,1},x_{2,2} \right \rangle$, under the following commutation relations: $x_{i,j}x_{i',j'}=x_{i',j'}x_{i,j}$, for all $1 \leq i,i' \leq 2$ and $0\leq j,j' \leq 2$. On the other hand, for $a^rz \in R$, $x_{i,j}a^rz=\sigma_{i,j}(a^rz)x_{i,j}=(a^r)^ia^jzx_{i,j}=a^{ri+j}zx_{i,j}$, where $a^{ri+j} \in \mathbb{F}_4$, $1\leq i \leq 2$ and $1\leq r,j \leq 2$. Of course, this example can be extended to any finite field $\mathbb{F}_{p^n}$ with $p$ a prime number. 

Let us show that $R$ is a $\Sigma$-skew
RNP ring. Note first that $\sigma_{1,0}$ is the identity homomorphism over $R$ and $\Sigma$ is closed under composition, that is, $\sigma_{i,j}^{\alpha}\in \Sigma$, for all $\alpha \in \mathbb{N}^6$. Notice that the set of nilpotent elements of $R$ is the ideal generated by $z$, that is, $N(R)=\left \langle z \right \rangle$. Let $f,g \in N(R)$ such that $fRg=0$. Notice that $f=a^rz$ and $g=a^sz$, for some $0\leq r,s \leq 2$. Futhermore, $\sigma_{i,j}^{\alpha}(f)=a^kz$, for all $\alpha \in \mathbb{N}^6$ and some $0\leq k \leq 2$. Hence, we have $gR\sigma_{i,j}^{\alpha}(f)=0$, for all $\alpha \in \mathbb{N}^6$ with $1\leq i \leq 2$ and $0\leq j \leq 2$. In this way, $R$ is a right $\Sigma$-skew RNP ring. 
\end{example}
%%%%%%%%%%%%%%%%%%%%%%%%%%%%%%%%%%%%%%%%%
Due to \cite[Proposition 1.8]{Bhattacharjee2020}, $R$ is $\sigma$-skew RNP if and only if it is $\sigma^{2k-1}$-skew RNP, for all $k\geq 1$. Note that if $R$ is $\sigma$-compatible  then $R$ is $\sigma$-skew RNP if and only if it is $\sigma^{k}$-skew RNP for all $k\geq 1$, as it is shown below. Since $\sigma(N(R))\subseteq N(R)$ for every endomorphism $\sigma$ of $R$ and since right (resp., left) $\sigma$-skew reflexive rings are right (resp., left) $\sigma$-skew RNP, Proposition \ref{prop-composkewReflex} implies immediately the following:
%%%%%%%%%%%%%%%%%%%%%%%%%%%%%%%%%%%%%%%
\begin{corollary}\label{cor-composkewRNP}
Let $\sigma$ be an endomorphism of $R$.
\begin{enumerate}
\item[\rm (1)] Composition of  finite family of compatible left (resp. right) skew
NRP endomorphisms is a left (resp. right) skew NRP endomorphism.
\item[\rm (2)] If $R$ is $\sigma$-compatible  then $R$ is right (resp.,
left) $\sigma$-skew NRP if and only if it is right (resp., left)
$\sigma^{k}$-skew NRP, for
all $k\geq 1$. Thus, if $R$ is $\sigma$-compatible then $R$
is $\sigma$-skew NRP if and only if it is $\sigma^{k}$-skew NRP, for
all $k\geq 1$.
\end{enumerate}
\end{corollary}
%%%%%%%%%%%%%%%%%%%%%%%%%%%%%%%%%%%%%%%
\begin{proposition}\label{prop-SigmaRNPiffsigma}
If $R$ is a $\Sigma$-compatible ring, then $R$ is left (resp., right) $\Sigma$-skew RNP if and only if it is left (resp., right) $\sigma_k$-skew RNP for each $\sigma_k\in \Sigma$.
\end{proposition}
%%%%%%%%%%%%%%%%%%%%%%%%%%%%%%%%%%%%%%%%%%%%%
\begin{proof}
The proof is similar to that of Proposition \ref{prop-Sigmarefiffsigma}, but using Remark \ref{rem-SigmaN(R)contN(R)} and Corollary \ref{cor-composkewRNP}.
\end{proof}
%%%%%%%%%%%%%%%%%%%%%%%%%%%%%%%%%%%%%%%%%%%%%
\begin{proposition}\label{prop-RNRiifSigmaRNP} If $R$ is a $\Sigma$-compatible ring, then the following assertions are equivalent:
\begin{enumerate}
\item[\rm (1)] $R$ is RNP.
\item[\rm (2)] $R$ is right  $\Sigma$-skew RNP.
\item[\rm (3)] $R$ is left  $\Sigma$-skew RNP.
\end{enumerate}
\end{proposition}
\begin{proof}
(1) $\Rightarrow$ (2) Suppose that $R$ is RNP and consider $a, b\in N(R)$ such that $aRb = 0$, that is, $arb = 0$, for any $r\in R$. By the $\Sigma$-compatibility of $R$, $\sigma^{\alpha}(a)rb=0$, i.e. $\sigma^{\alpha}(a)Rb=0$ (Lemma \ref{colosss}). Since $a\in N(R)$, $\sigma^{\alpha}(a)\in \sigma^{\alpha}(N(R))$, and by Remark \ref{rem-SigmaN(R)contN(R)}, we have $\sigma^{\alpha}(a)\in N(R)$. As $R$ is RNP, $\sigma^{\alpha}(a)Rb=0$ implies that $bR\sigma^{\alpha}(a)= 0$, whence $R$ is right $\Sigma$-skew RNP.

(2) $\Rightarrow$ (1) Let $R$ be right $\Sigma$-skew RNP and let $a, b\in N(R)$ such that $aRb = 0$, i.e. $arb = 0$, for any $r\in R$. Since $R$ is $\Sigma$-compatible, $ar\sigma^{\alpha}(b) = 0$, for every $\alpha\in \mathbb{N}^n$. Using that $b\in N(R)$, we obtain $\sigma^{\alpha}(b)\in \sigma^{\alpha}(N(R))$, and by Remark \ref{rem-SigmaN(R)contN(R)}, $\sigma^{\alpha}(b)\in N(R)$. Now, as $R$ is right $\Sigma$-skew RNP, then $\sigma^{\alpha}(b)R\sigma^{\alpha}(a) = 0$. So, for each $r\in R$, $\sigma^{\alpha}(bra) =
\sigma^{\alpha}(b)\sigma^{\alpha}(r)\sigma^{\alpha}(a) = 0$. Since $\Sigma$ is compatible, then
$\sigma^{\alpha}$ is injective, and so $bra = 0$, for all $r\in R$. Therefore $bRa=0$, and hence $R$ is RNP.

With a reasoning similar to the previous one, one can prove (1) $\Leftrightarrow$ (3).
\end{proof}
%%%%%%%%%%%%%%%%%%%%%%%%%%%%%%%%%%%%%%%%%%%%%
\begin{corollary}(\cite[Proposition 1.7]{Bhattacharjee2020}) Let $R$ be a $\sigma$-compatible ring. Then the following are equivalent: (1) $R$ is RNP (2) $R$ is right  $\sigma$-skew RNP (3) $R$ is left  $\sigma$-skew RNP.
\end{corollary}
%%%%%%%%%%%%%%%%%%%%%%%%%%%%%%%%%%%%%%%
\begin{corollary}\label{cor-RefliifSigmaRefl} If $R$ is a $\Sigma$-compatible ring, then the following are equivalent:
\begin{enumerate}
\item[\rm (1)] $R$ is reflexive.
\item[\rm (2)] $R$ is right  $\Sigma$-skew reflexive.
\item[\rm (3)] $R$ is left  $\Sigma$-skew reflexive.
\end{enumerate}
\end{corollary}
%%%%%%%%%%%%%%%%%%%%%%%%%%%%%%%%%%%%%%%%%%%%%
\begin{remark}\label{rem-nilimplRNP}
It is straightforward to prove that nil-reversible rings are RNP.
\end{remark}
%%%%%%%%%%%%%%%%%%%%%%%%%%%%%%%%%%%%%%%%%%%
The following property relates nil-reversible and
$\Sigma$-compatible rings with $\Sigma$-skew RNP rings.
%%%%%%%%%%%%%%%%%%%%%%%%%%%%%%%%%%%%%%%%%%%
\begin{proposition}\label{prop-nil-revimplSigma}
If $R$ is  nil-reversible and $\Sigma$-compatible then $R$ is $\Sigma$-skew RNP.
\end{proposition}
\begin{proof}
Let $\Sigma=\{\sigma_1,\dots,\sigma_n\}$ be a finite family of endomorphisms of a nil-reversible and $\Sigma$-compatible ring $R$. Suppose that $aRb=0$, for $a,b\in N(R)$. By Remark \ref{rem-nilimplRNP}, $bRa=0$, and so $bR\sigma^{\alpha}(a)=0$, since $R$ is $\Sigma$-compatible. Thus $R$ is right $\Sigma$-skew RNP. From Lemma \ref{colosss}, we have that $bRa=0$ implies $\sigma^{\alpha}(b)Ra=0$, and therefore $R$ is left $\Sigma$-skew RNP. 
\end{proof}
%%%%%%%%%%%%%%%%%%%%%%%%%%%%%%%%%%%%%%%%%%%%%
Due to \cite[Proposition 3.2.1]{Fajardoetal2020}, a skew PBW extension $A$ over a domain $R$ is also a domain, and so $N(A)=0$. In this way, if $A$ is a skew PBW extension over a domain $R$ which satisfies the conditions established in Proposition \ref{prop-indicsigma}, then $A$ is $\overline{\Sigma}$-skew RNP for
$\overline{\Sigma}=\{\overline{\sigma_1}, \dots, \overline{\sigma_n} \}$, where $\overline{\sigma_k}$ is as in Proposition \ref{prop-indicsigma}.
%%%%%%%%%%%%%%%%%%%%%%%%%%%%%%%%%%%%%%%%%%%%%
\begin{proposition}\label{prop-Sigmarigid}
If $A=\sigma(R)\langle x_1, \dots, x_n\rangle$ is skew PBW extension over a $\Sigma$-rigid ring $R$, then the following assertions hold:
\begin{enumerate}
\item[\rm (1)] Both $R$ and $A$ are reflexive.
\item[\rm (2)] $R$ is $\Sigma$-skew RNP.
\item[\rm (3)] If the conditions established in Proposition \ref{prop-indicsigma} hold, then $A$ is  $\overline{\Sigma}$-skew RNP.
\end{enumerate}
\end{proposition}
\begin{proof}
(1) By \cite[Theorem 4.4]{ReyesSuarez2018-3}, we have that both $R$ and $A$ is reduced. Since reduced rings are reflexive, then $R$ and $A$ are
reflexive.

(2) As $R$ is reduced, $R$ is $\Sigma$-skew RNP for any finite family of endomorphisms $\Sigma$.

(3) Since $A$ is reduced, it follows that $A$ is $\Sigma$-skew RNP for any finite family of endomorphisms. In particular, $A$ is  $\overline{\Sigma}$-skew RNP for
$\overline{\Sigma}=\{\overline{\sigma_1}, \dots, \overline{\sigma_n}
\}$, where $\overline{\sigma_k}$ is as in Proposition \ref{prop-indicsigma}.
\end{proof}
%%%%%%%%%%%%%%%%%%%%%%%%%%%%%%%%%%%%%%%%%%%%%
Following Kwak and Lee \cite{KwakLee2011}, a ring $R$ is called {\em CN} if $N(R[x])\subseteq N(R)[x]$. Both NI rings and Armendariz rings are CN but not conversely. The classes of quasi-Armendariz and CN rings are independent of each other (Kheradmand et al. \cite[Examples 3.15(1) and (2)]{Kheradmandetal2017}). Bhattacharjee \cite{Bhattacharjee2020} showed that if a ring $R$ is $CN$ and quasi-Armendariz, then there is an equivalence of the RNP property between the ring $R$ and its ring of polynomials with coefficients in $R$. 

\vspace{0.2cm}

Motivated by the previous facts, we introduce the following notion with the aim of study the RNP property on Ore extensions and skew PBW extensions.
%%%%%%%%%%%%%%%%%%%%%%%%%%%%%%%%%%%%
\begin{definition}\label{Def-CN}
Let $A = \sigma(R)\langle x_1,\dotsc, x_n\rangle$ be a skew PBW extension over $R$. $R$ is said to be $\Sigma$-\emph{skew CN} if $N(A) \subseteq N(R)A$, where $N(R)A:=\{f=\sum_{i=0}^ta_iX_i\in A\mid a_i\in N(R)$, for  $0\leq i\leq
t\}$.
\end{definition}
%%%%%%%%%%%%%%%%%%%%%%%%%%%%%%%
Hashemi et al. \cite{HashemiKhalilAlhevaz2019}, studied the connections of the prime radical and the upper nil radical of $R$ with the prime radical and the upper nil radical of the extension skew PBW. For instance, some results establish sufficient conditions to guarantee that $N^{*}(A)\subseteq N^{*}(R)A$, for a bijective skew PBW extension $A$ over a $(\Sigma,\Delta)$-compatible ring $R$ \cite[Theorem 3.15]{HashemiKhalilAlhevaz2019}. In this way, if $A$ is NI, it follows that $R$ is $\Sigma$-skew CN. Other important result states that if $A$ is a PBW extension over a $\Sigma$-compatible ring $R$ and $N(R)$ is a $\Delta$-ideal, then $N(A)\subseteq N(R)A$ \cite[Proposition 4.1]{HashemiKhalilAlhevaz2019}, that is, $R$ is a $\Sigma$-skew CN. From \cite[Corollary 4.12]{HashemiKhalilAlhevaz2019}, it follows that if $R$ is a 2-primal $\Sigma$-compatible ring, then $R$ is $\Sigma$-skew CN.

\vspace{0.2cm}

Reyes and Su\'arez \cite{ReyesSuarez2019-2} defined and studied the weak $(\Sigma,\Delta)$-compatible property for skew PBW extensions. In particular, it was proved that if $A$ is a skew PBW extension over a weak $(\Sigma,\Delta)$-compatible NI ring $R$, then $f=\sum_{i=0}^ta_iX_i\in N(A)$ if and only if $a_i\in N(R)$, for $0\leq i\leq t$ \cite[Theorem 4.6]{ReyesSuarez2019-2}. Thus, if $f\in N(A)$, then $a_i\in N(R)$, for $0\leq i\leq t$, whence $f\in N(R)A$. This means that if $A$ is a skew PBW extension over a weak $(\Sigma,\Delta)$-compatible NI ring $R$, then $A$ is $\Sigma$-skew CN.
%%%%%%%%%%%%%%%%%%%%%%%%%%%%%%%%%%%%
\begin{example} We present some examples of $\Sigma$-skew CN rings.
\begin{itemize}
    %\item[(i)] Let $D_h = A_1(\Bbbk)[x_h;\sigma_h]$ be the skew polynomial ring over $\Bbbk[t][x;\frac{\partial}{\partial t} ]$ known as {\em the mixed algebra}. We can verify that $N(D_h)\subseteq N(A_1(\Bbbk))[x_h;\sigma_h]$.
      \item[(1)] Let $q, h \in \Bbbk$, $q \neq 0$. Consider the algebra of {\em q-differential operators} $D_{q,h}[x, y]=\Bbbk[y][x, \sigma, \delta]$ with $\sigma(y)=qy$ and $\delta(y)=h$. This class of algebras satisfies that $N(D_{q,h}[x, y])\subseteq N(\Bbbk[y])[x, \sigma,\delta]$.
      Since Ore extensions of injective type are skew PBW extensions, we have other more general examples of Ore extensions that
      are $\Sigma$-skew CN: if $R$ is a $\sigma$-compatible ring and $N(R)$ a $\delta$-ideal of $R$, then by \cite[Proposition 2.2]{Nasr-Isf2014}, $N(R[x;\sigma,\delta]) \subseteq N(R)[x;\sigma,\delta]$, i.e. $R$ is $\Sigma$-skew CN. If $R[x;\sigma, \delta]$ is NI and $N(R)$ is $\sigma$-rigid then by \cite[Theorem 3.1]{Nasr-Isf2015}, we have that $N(R[x;\sigma, \delta]) = N(R)[x;\sigma,\delta]$, so $R$ is $\Sigma$-skew CN.
      \item[(2)] If $R$ is a commutative ring and $\mathfrak{g}$ is a finite dimensional Lie algebra over $R$ with basis $\left \{x_1,\dots,x_n \right \}$, then the {\it universal enveloping algebra} of $\mathfrak{g}$, denoted $U(\mathfrak{g})$, is a PBW extension over $R$. The set of nilpotent elements of the enveloping algebra satisfies  $N(A)=N(U(\mathfrak{g}))\cong N(\sigma(R)\left \langle x_1, \dots, x_n \right \rangle)\subseteq N(R)A$.
      \item[(3)] Consider the ring $A$ presented in Example \ref{ejemploraro}. Let $f \in A$, such that $f^l=0$, for some $l \in \mathbb{N}$, where $f=a_mX_m + \cdots+ a_1X_1+a_0 = a_mX_m + h$ with $deg(X_i)=\alpha_i$ and $h=a_{m-1}X_m + \cdots + a_0$. Each $X_i$ is an element of the set $\left \{ \prod_{i,j,k}X_{i,j}^{\alpha_k} \ | \ 1\leq k \leq 6 \right \}$. Thus, we have that $0=f^l=(a_mX_m+ \Delta)^l =  a_m\prod_{k=1}^{l-1}\sigma_{i,j}^{k\alpha_m}(a_m)X_m+\ {\rm lower\ terms} $. Since $a_m\prod_{k=1}^{l-1}\sigma_{i,j}^{k\alpha_m}(a_m)=0$, it follows that $a_m^l=0$, by using $\Sigma$-compatibility. Hence, we obtain $a_m \in N(R)=\left \langle z \right \rangle$. On the other hand, we can observe that $0=f^l=(a_mX_m + h)^l= Q + h^l$, where $h=a_{m-1}X_m + \cdots + a_0$ and each coefficient of $Q$ can be expressed in $a_m$ and $\sigma_{i,j}^{k\beta}(a_m)$, for some $\beta \in \mathbb{N}^6$ with $k \in \mathbb{N}$. From the above, it follows that $h^l \in N(R)\left \langle x_{1,0},x_{1,1},x_{1,2},x_{2,0},x_{2,1},x_{2,2}\right \rangle$. Finally, if $h={m-1}X_m + \cdots + a_0$, then $h^l=(a_{m-1}X_m + \cdots + a_0)^l= a_{m-1}\prod_{k=1}^{l-1}\sigma_{i,j}^{k\alpha_{m-1}}(a_{m-1})X_{m-1}+\ {\rm lower\ terms}$. By the same previous argument, we have $a_{m-1}\in N(R)$. By continuing this process, we can see that $a_i \in N(R)=\left \langle z \right \rangle$, for all $0 \leq i \leq m$. Therefore, we have $f \in N(R)\left \langle x_{1,0},x_{1,1},x_{1,2},x_{2,0},x_{2,1},x_{2,2}\right \rangle$, that is, $R$ is a $\Sigma$-skew CN ring.
\end{itemize}
\end{example}
%%%%%%%%%%%%%%%%%%%%%%%%%%%%%
\begin{theorem}\label{AskewRNP} Let $A=\sigma(R)\langle
x_1,\dots,x_n\rangle$ be a skew PBW extension over a $(\Sigma,\Delta)$-compatible and $\Sigma$-skew CN ring $R$ that satisfies the {\rm (}SQA1{\rm )} condition. If the conditions established in Proposition \ref{prop-indicsigma} hold, then $R$ is right $\Sigma$-skew RNP if and only if $A$ is right $\overline{\Sigma}$-skew reflexive RNP.
\end{theorem}
\begin{proof}
Let $R$ be a right $\Sigma$-skew RNP ring and consider two nilpotent elements of $A$, that is, $f = \sum_{i=0}^t a_i X_i$, $g = \sum_{j=0}^m b_jY_j \in N(A)$ with ${\rm deg}(X_i)=\alpha_i$, for $1 \leq i \leq t$ and ${\rm deg}(Y_j)=\beta_j$, for $1 \leq j \leq m$, respectively, such that $fAg = 0$. Since $R$ satisfies (SQA1) condition, then $a_iRb_j = 0$, for all $i, j$. On the other hand, $R$ is $\Sigma$-skew CN, that is, $f,g \in N(R)A$, which implies that $a_i, b_j \in N(R)$ for all $i, j$. By using the $\Sigma$-skew RNP property of $R$, we have $b_j R\sigma^{\alpha}(a_i) = 0$, for all $i, j$ and $\alpha \in \mathbb{N}^n\ \backslash\ \{0\}$. Consider an element $h\in A$. We can note that each coefficient of $gh\overline{\sigma}^{\alpha}(f)$ are products of the coefficient $b_j$ with elements of $R$ and several evaluations of $a_i$ in $\sigma$'s and $\delta$'s depending of the coordinates of $\alpha_i, \beta_j$. From the previous statement and the fact that $b_j R\sigma^{\alpha}(a_i) = 0$, by using of the $(\Sigma,\Delta)$-compatibility of $R$, it follows that  $gh\overline{\sigma}^{\alpha}(f) = 0$, for all $\alpha \in  \mathbb{N} \ \backslash\ \{0\}$. Since $h$ is an arbitrary element of $A$, we have $gA\overline{\sigma}^{\alpha}(f) = 0$. Hence, $A$ is right $\overline{\Sigma}$-skew RNP. 

Conversely, suppose that $A$ is a right $\overline{\Sigma}$-skew RNP ring and $a,b \in N(R)$ such that $aRb=0$. Since $R$ is $(\Sigma,\Delta)$-compatible, we have $aAb=0$. This means that $bA\overline{\sigma}^{\alpha}(a)=0$, for all $\alpha \in \mathbb{N}^n$ entailing $bR\sigma^{\alpha}(a)=0$, by using RNP property of $A$ and $\overline{\sigma}^{\alpha}(a)\in N(R)$. Therefore, $R$ is right $\Sigma$-skew RNP.
\end{proof}
%%%%%%%%%%%%%%%%%%%%%%%%%%%%%%%%%%%%%%%%
\begin{corollary}(\cite[Corollary 2.13]{Bhattacharjee2020}) If $R$ is an Armendariz ring with an endomorphism $\sigma$, then, $R$ is right $\sigma$-skew RNP if and only if $R[x]$ is right $\overline{\sigma}$-skew RNP.
\end{corollary}
%%%%%%%%%%%%%%%%%%%%%%%%%%%%%%%%%%%%%%%%
\begin{corollary}
Let $R[x;\sigma]$ be a skew polynomial ring and $\sigma'$ be an endomorphism of $R$. If $R$ satisfies the (SQA1) condition, is skew CN and $\sigma$-compatible, then $R$ is right $\sigma'$-skew RNP if and only if $R[x;\sigma]$ is right $\overline{\sigma'}$-skew RNP.
\end{corollary}
%%%%%%%%%%%%%%%%%%%%%%%%%%%%%%%%%%%%%%%%
\begin{corollary}(\cite[Proposition 2.12]{Bhattacharjee2020}) Let $R$ be a ring with an endomorphism $\sigma$. Suppose that $R$ is quasi-Armendariz and CN. Then, $R$ is right $\sigma$-skew RNP if and only if $R[x]$ is right $\overline{\sigma}$-skew RNP.
\end{corollary}
%%%%%%%%%%%%%%%%%%%%%%%%%%%%%%%%%%%%%%%%
\begin{proposition}\label{prop-RNP}
If $A$ is a skew PBW extension over a $(\Sigma,\Delta)$-compatible NI ring $R$, then $A$ is nil-reflexive if and only if $A$ is RNP.
\end{proposition}
\begin{proof}
Let $A$ be a skew PBW extension over $R$. Notice that $R$ is weak $(\Sigma,\Delta)$-compatible. Thus, if $R$ is NI, from \cite[Theorem 3.3]{SuarezChaconReyes2021} we obtain that $A$ is NI, whence $N^{*}(A)=N(A)$. By \cite[Proposition 2.1]{Kheradmandetal2017(b)}, we
have that $A$ is  nil-reflexive if and only if $fAg= 0$ implies $gAf = 0$, for elements $f, g \in N^{*}(A)$. Hence, $fAg= 0$ implies $gAf = 0$, for elements $f, g \in N(A)$. Therefore, $A$ is nil-reflexive if and only if $A$ is  RNP.
\end{proof}
%%%%%%%%%%%%%%%%%%%%%%%%%%%%%%%%%%%%%%%%%
\begin{corollary}\label{cor-nilrefiffSigmaRNP} Let $A$ be a skew PBW extension over an NI ring $R$ and suppose that the conditions established in Proposition \ref{prop-indicsigma} hold. If $A$ is $(\overline{\Sigma},\overline{\Delta})$-compatible, then $A$ is nil-reflexive if and only if $A$ is $\overline{\Sigma}$-skew RNP.
\end{corollary}
%%%%%%%%%%%%%%%%%%%%%%%%%%%%%%%%%%%%%%%%%
\begin{proof}
The assertion follows from Propositions \ref{prop-RNRiifSigmaRNP} and \ref{prop-RNP}.
\end{proof}
%%%%%%%%%%%%%%%%%%%%%%%%%%%%%%%
Next, we recall the following result which describes the Ore localization by regular elements of the ring $R$ over a skew PBW extension.
%%%%%%%%%%%%%%%%%%%%%%%%%%%%
\begin{proposition} (\cite[Lemma 2.6]{OreGoldieLezama2013}) \label{Lemma2.6}
Let $A = \sigma(R)\langle x_1,\dotsc, x_n\rangle$ be a skew PBW extension over $R$, and let $S$ be the set of regular elements of $R$ such that $\sigma_i(S) = S$, for every $1 \leq i \leq n$, where $\sigma_i$ is defined by Proposition \ref{sigmadefinition}.
\begin{itemize}
    \item[{\rm (1)}] If $S^{-1}R$ exists, then $S^{-1}A$ exists and it is a bijective skew PBW extension over $S^{-1}R$ with $S^{-1}A = \sigma(S^{-1}R)\left \langle x_1', \dots , x_n' \right \rangle$, where $x_i':= \frac{x_i}{1}$ and the system of constants of $S^{-1}R$ is given by $d_{i,j}':=\frac{d_{i,j}}{1}$, $c_{i,\frac{r}{s}}':=\frac{\sigma_i(r)}{\sigma_i(s)}$, for all $1 \leq i, j \leq n$.       
\item[{\rm (2)}] If $RS^{-1}$ exists, then $AS^{-1}$ exists and it is a bijective skew PBW extension over $RS^{-1}$ with $AS^{-1} = \sigma(RS^{-1})\left \langle x_1', \dots , x_n' \right \rangle$, where $x_i'':= \frac{x_i}{1}$ and the system of constants of $S^{-1}R$ is given by $d_{i,j}'':=\frac{d_{i,j}}{1}$, $c_{i,\frac{r}{s}}'':=\frac{\sigma_i(r)}{\sigma_i(s)}$, for all $1 \leq i, j \leq n$. 
\end{itemize}
\end{proposition}
%%%%%%%%%%%%%%%%%%%%%%%%%%%%%%%%%%%%%%%%%%%%%
Related with Proposition \ref{Lemma2.6}, if $S$ is a multiplicatively closed subset of a ring $R$ consisting of central regular elements, and $\sigma$ is an automorphism of $R$ such that $\sigma(S) \subseteq S$, then the mapping $ \overline{\sigma}: S^{-1}R \rightarrow S^{-1}R$ defined by $\overline{\sigma}(u^{-1}a) = \sigma(u)^{-1}\sigma(a)$ for $u \in S$ and $a \in R$, induces an endomorphism of $S^{-1}R$. If $\Sigma$ is a set of automorphisms over $R$, we denote $\Sigma_S$ the set of automorphisms over $S^{-1}R$, induced by $\Sigma$.
%%%%%%%%%%%%%%%%%%%%%%%%%%%%%%%%%%%%%%%%%
\begin{theorem}\label{localizationRNP}
Let $R$ be a ring with a finite family of automorphism $\Sigma$ and $S$ a multiplicatively closed subset of $R$ consisting of central regular elements such that $\sigma^{\alpha}(S) \subseteq S$, for every $\alpha\in \mathbb{N}^{n}$. Then, $R$ is right $\Sigma$-skew RNP if and only if $S^{-1}R$ is right $\Sigma_S$-skew RNP.
\end{theorem}
\begin{proof}
Suppose that $R$ is a right $\Sigma$-skew RNP and let $s_1^{-1}a,s_2^{-1}b \in N(S^{-1}R)$ such that $(s_1^{-1}a)S^{-1}R(s_2^{-1}b)=0$, that is, $(s_1^{-1}a)s^{-1}r(s_2^{-1}b)=0$, for all $s^{-1}r \in S^{-1}R$. Notice that $N(S^{-1}R)=S^{-1}N(R)$, which implies that $a,b \in N(R)$. In addition, $(s_1ss_2)^{-1}(arb)=(s_1^{-1}a)s^{-1}r(s_2^{-1}b)=0$. By definition of $S^{-1}R$, there is $c\in S$ such that $c$ is a central element of $R$ and $0=(arb)c=a(rc)b$. Since $R$ is right $\Sigma$-skew RNP, it follows that $b(rc)\sigma^{\alpha}(a)=0$, for all nonzero $\alpha$ in $\mathbb{N}^n$. Therefore, we have that $(s_2^{-1}b)(s^{-1}r)\overline{\sigma}^{\alpha}(s_1^{-1}a)=(s_2^{-1}b)(s^{-1}r)\sigma^{\alpha}(s_1)^{-1}\sigma^{\alpha}(a)=(s_2s\sigma^{\alpha}(s_1))^{-1}(br\sigma^{\alpha}(a))=0$, since $b(r)\sigma^{\alpha}(a)c=0$, for each nonzero $\alpha$ in $\mathbb{N}^n$. Therefore, $S^{-1}R$ is right $\Sigma_S$-skew RNP. Reciprocally, let $a,b \in N(R)$ such that $aRb=0$, that is, $arb=0$, for all $r\in R$. We have that $(1^{-1})(arb)=(1^{-1}a)(1^{-1}r)(1^{-1}b)=0$. If $S^{-1}R$ is right $\Sigma_S$-skew RNP, we have $0=(1^{-1}b)(1^{-1}r)\overline{\sigma}^{\alpha}(1^{-1}a)=(1^{-1}b)(1^{-1}r)\sigma^{\alpha}(1)^{-1}\sigma^{\alpha}(a)=(1^{-1})(br\sigma^{\alpha}(a))$. Thus, $br\sigma^{\alpha}(a)=0$, that is, $R$ is right $\Sigma$-skew RNP.    
\end{proof}
%%%%%%%%%%%%%%%%%%%%%%%%%%%%%%%%%%%%%%%%%%%%%%%%
\begin{theorem}
Let $A = \sigma(R)\langle x_1,\dotsc, x_n\rangle$ be a bijective skew PBW extension over a $(\Sigma,\Delta)$-compatible ring $R$. Suppose that the conditions established in Proposition \ref{prop-indicsigma} hold, and let $S$ be a multiplicatively closed subset of $R$ consisting of central regular elements such that $\overline{\sigma}^{\alpha}(S) \subseteq S$ and $S^{-1}R$ is a $\overline{\Sigma}_S$-skew CN which satisfies the (SQA1) condition. Then, $A$ is right $\overline{\Sigma}$-skew RNP if and only if $S^{-1}A$ is right $\overline{\Sigma}_S$-skew RNP.
\end{theorem}
\begin{proof}
Suppose $A$ is right $\overline{\Sigma}$-skew RNP. Since $S^{-1}R$ is a $\Sigma_S$-skew CN ring, then $R$ is $\Sigma$-skew CN. Similarly, the condition (SQA1) also transfers from $S^{-1}R$ to $R$. Thus, we have that $R$ is right $\Sigma$-skew RNP, by using Theorem \ref{AskewRNP}. By hypothesis, $S$ be a multiplicatively closed subset of $R$ consisting of central regular elements, which means that $S^{-1}R$ is right $\Sigma_S$-skew RNP, by Theorem \ref{localizationRNP}. In addition, by Proposition \ref{Lemma2.6}, we have that $S^{-1}A$ is a bijective skew PBW extension over $S^{-1}R$, that is, $S^{-1}A = \sigma(S^{-1}R)\left \langle x_1', \dots , x_n' \right \rangle$. Since $R$ is $(\Sigma, \Delta)$-compatible, then $S^{-1}R$ is $(\overline{\Sigma},\overline{\Delta})$-compatible \cite[Theorem 4.20]{ReyesSuarez2018RUMA}. Notice that $S^{-1}A$ is a bijective skew PBW extension over $S^{-1}R$, where $S^{-1}R$ is a $(\overline{\Sigma},\overline{\Delta})$-compatible $\overline{\Sigma}_S$-skew CN ring which satisfies the (SQA1) condition. Hence, since $S^{-1}R$ is right $\Sigma$-skew RNP, then $S^{-1}A$ is a right $\overline{\Sigma}_S$-skew RNP ring, by Theorem \ref{AskewRNP}. Reciprocally, if $S^{-1}A$ is right $\overline{\Sigma}_S$-skew RNP, where $S^{-1}A = \sigma(S^{-1}R)\left \langle x_1', \dots , x_n' \right \rangle$, then $S^{-1}R$ is right $\overline{\Sigma}_S$-skew RNP, because $S^{-1}R$ is $(\overline{\Sigma},\overline{\Delta})$-compatible $\overline{\Sigma}_S$-skew CN and satisfies the condition (SQA1). The above statement follows from Theorem \ref{AskewRNP}. Since $S^{-1}R$ is $\overline{\Sigma}_S$-skew RNP, then $R$ is $\Sigma$-skew RNP. 
Finally, $R$ is a $(\Sigma, \Delta)$-compatible ring and inherits the $\Sigma$-skew CN property and the condition (SQA1) from the ring $S^{-1}R$, which implies that $A$ is a right $\overline{\Sigma}$-skew RNP ring, by using Theorem \ref{AskewRNP}.      
\end{proof}
%%%%%%%%%%%%%%%%%%%%%%%%%%%%%%%%%%%%%%%%%%%%%%%%
%%%%%%%%%%%%%%%%%%%%%%%%%%%%%%%%%%%%%%%%%%%%%%%%
\section*{Acknowledgments}
Research is supported by Grant HERMES CODE 52464, Department of Mathematics, Faculty of Science, Universidad Nacional de Colombia - Sede Bogot\'a.

\end{document}